\documentclass[12pt, twoside]{article}
\usepackage{amsmath,amsthm,amssymb, amsfonts, nccmath}
\usepackage{times}
\usepackage{enumerate}
\usepackage{tikz}
\usepackage{graphicx}

\def\titlerunning#1{\gdef\titrun{#1}}
\makeatletter
\def\author#1{\gdef\autrun{\def\and{\unskip, }#1}\gdef\@author{#1}}
\def\address#1{{\def\and{\\\hspace*{18pt}}\renewcommand{\thefootnote}{}%
\footnote {#1}}%
\markboth{\autrun}{\titrun}}
\makeatother
\def\email#1{e-mail: #1}
\def\subjclass#1{{\renewcommand{\thefootnote}{}%
\footnote{\emph{Mathematics Subject Classification (2010):} #1}}}
\def\keywords#1{\par\medskip
\noindent\textbf{Keywords.} #1}

\newtheorem{theorem}{Theorem}[section]
\newtheorem*{theorem*}{Theorem}

\newtheorem{prop}[theorem]{Proposition}

\newtheorem{obs}[theorem]{Observation}
\theoremstyle{definition}

\newtheorem{remark}[theorem]{Remark}
\newtheorem*{remark*}{Remark}

\theoremstyle{definition}

\numberwithin{equation}{section}

\newcommand{\conv}{\mbox{conv}}

\newcommand{\fts}{\footnotesize}

\begin{document}


\baselineskip=17pt


\titlerunning{Two extensions of the Erd\H os--Szekeres problem}

\title{Two extensions of the Erd\H os--Szekeres problem}

\author{
Andreas F. Holmsen 
\and 
Hossein Nassajian Mojarrad
\and
J\'{a}nos Pach
\and
G\'abor Tardos
}

\date{}

\maketitle

\address{
A.~F.~Holmsen: Department of Mathematical Sciences, KAIST,
Daejeon; \email{andreash@kaist.edu}
\and
H.~N.~Mojarrad: Courant Institute, New York University, New York;
\email{sn2854@nyu.edu}
\and
J. Pach: R{\'e}nyi Institute, Budapest and 
Moscow Institute of Physics and Technology, Moscow; \email{pach@cims.nyu.edu}
\and
G. Tardos: R{\'e}nyi Institute, Budapest and Department of Mathematics, Central European University, Budapest; \email{tardos@renyi.hu}
}



\begin{abstract}
According to Suk's breakthrough result on the Erd\H os--Szekeres problem, any point set in general position in the plane, which has no $n$ elements that form the vertex set of a convex $n$-gon, has at most $2^{n+O\left({n^{2/3}\log n}\right)}$ points. We strengthen this theorem in two ways. First, we show that the result generalizes to convexity structures induced by pseudoline arrangements. Second, we improve the error term.

A family of $n$ convex bodies in the plane is said to be in {\em convex position} if the convex hull of the union of no $n-1$ of its members contains the remaining one. If any {\em three} members are in convex position, we say that the family is in {\em general position}. Combining our results with a theorem of Dobbins, Holmsen, and Hubard, we significantly improve the best known upper bounds on the following two functions, introduced by Bisztriczky and Fejes T\'oth and by Pach and T\'oth, respectively. Let $c(n)$ (and $c'(n)$) denote the smallest positive integer $N$ with the property that any family of $N$ pairwise disjoint convex bodies in general position (resp., $N$ convex bodies in general position, any pair of which share at most two boundary points) has an $n$-membered subfamily in convex position. We show that $c(n)\le c'(n)\leq 2^{n+O\left(\sqrt{n\log n}\right)}$.
\end{abstract}

\section{Introduction}
We say that a set of $n$ points in the plane is in {\em convex position} if the convex hull of no $n-1$ of them contains the $n$-th point. If no three elements of the set are collinear (that is, any three points are in convex position), then the set is said to be in {\em general position}. According to a classical conjecture of Erd\H os and Szekeres~\cite{erd-sze1}, if $P$ is a set of points in general position in the plane with $|P|\ge 2^{n-2}+1$, then it has $n$ elements in convex position. This bound, if true, cannot be improved~\cite{erd-sze2}. In a recent breakthrough, Suk~\cite{andrew} came close to proving the conjectured bound.

\begin{theorem} [Suk, 2017]\label{suk}
Given any integer $n\geq 3$, let $e(n)$ denote the smallest number with the property that every family of at least $e(n)$ points in general position in the plane has $n$ elements in convex position. Then we have
$$e(n)\leq 2^{n+O(n^{2/3}\log n)}.$$
\end{theorem}

\paragraph{\bf Pseudo-configurations.}
A set of simple continuous curves in the Euclidean plane that start and end ``at infinity'' is called an {\em arrangement of pseudolines} if any two of them meet in precisely one point: at a proper crossing. A {\em pseudo-configuration} is a finite set of points $P$ in the Euclidean plane such that each pair of distinct points $p$ and $q$ in $P$ span a unique {\em pseudoline}, denoted by $\ell(p,q)$ such that $L(P) = \{\ell(p,q) : p, q \in P, p\neq q\}$ form a pseudoline arrangement and for any $p,q\in P$, $p\ne q$ we have $\ell(p,q)=\ell(q,p)$ and $\ell(p,q)\cap P=\{p,q\}$; see \cite{goody}.

\begin{figure}\centering
\begin{tikzpicture}
\node at (0,0) {\includegraphics[height=6.4cm]{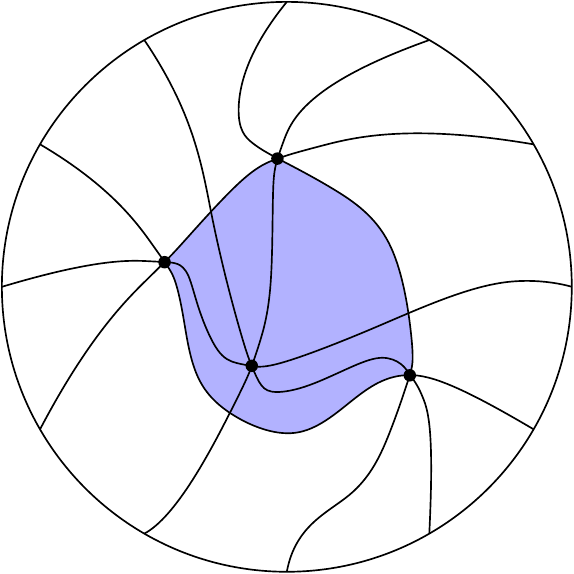}};
\node at (-1.43,0) {\fts $p$};
\node at (96:1.7cm) {\fts $q$};
\node at (1.6,-.8) {\fts $r$};
\node at (-.2,-.7) {\fts $s$};
\node at (37:2.5cm) {\tiny $\ell(p,q)$};
\node at (-4:2.5cm) {\tiny $\ell(p,s)$};
\node at (90:2.5cm) {\tiny $\ell(q,r)$};
\node at (128:2.5cm) {\tiny $\ell(r,s)$};
\node at (159:2.5cm) {\tiny $\ell(p,r)$};
\node at (252:2.5cm) {\tiny $\ell(q,s)$};
\node at (17:3.4cm) {\tiny $\ell_\infty$};
\end{tikzpicture}
\caption{A pseudo-configuration of four points with the convex hull shaded.}
\label{pseudo}
\end{figure}
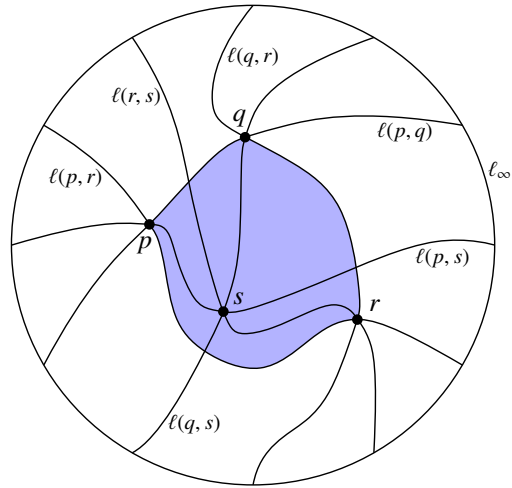

This underlying pseudoline arrangement induces a convexity structure on the point configuration in a natural way. For any pair of points $p, q\in P$, the bounded portion of $\ell(p,q)$ between $p$ and $q$ is called the {\em pseudosegment} connecting $p$ and $q$. If we delete from the plane all pseudosegments between the elements of $P$, the plane is divided into a number of connected components, precisely one of which is unbounded. The {\em convex hull} of the configuration is defined as the complement of the unbounded region, and is denoted by $\conv P$. We say that a subset $Q\subseteq P$ is in {\em convex position} if no point $p\in Q$ is in the convex hull of $Q\setminus \{p\}$.\footnote{Pseudo-configurations also have a purely combinatorial characterization. They can be defined by several equivalent systems of axioms. Other names for pseudo-configurations that can be found in the literature are {\em generalized configurations} \cite{goody}, {\em uniform rank 3 acyclic oriented matroids} \cite{OMS}, and {\em CC-systems} \cite{knuti}.}

It turns out that for four points there are only two combinatorially distinct pseudo-configurations and both can be obtained from straight lines, but for five or more points there exist pseudo-configurations that are not realizable by straight lines. Still, the number of possible pseudo-configurations on five points is finite and we will leave the verification of some simple statements about at most five points in a pseudo-configuration to the reader. (This applies in particular to Observations \ref{trivial}  and \ref{trivial2}, below.)

Many basic theorems of convexity hold in this more general setting. For instance, a set of points is in convex position if and only if every four of its elements are in convex position \cite{ragavan}. This is Carath\'{e}odory's theorem in the plane.

Goodman and Pollack \cite{goopolES} proposed the generalization of the Erd\H os-Szekeres problem to pseudo-configurations. The original ``cup-cap'' proof due to Erd{\H o}s and Szekeres \cite{erd-sze1} readily generalizes to this setting:

\begin{theorem} \label{weak pseudo ES}
Let $P$ be a pseudo-configuration. If $|P|\geq 4^n$, then $P$ contains an $n$-element subset in convex position.
\end{theorem}

The purpose of this note is to show that Suk's breakthrough result, Theorem~\ref{suk} carries over to pseudo-configurations. In the process we also improve on the error term.

\begin{theorem} \label{gen conf ES}
Given any $n\geq 3$, let $b(n)$ denote the smallest number such that every pseudo-configuration of size at least $b(n)$ has $n$ members in convex position. Then we have
\[b(n) \leq 2^{n+O\left(\sqrt{n\log n}\right)}.\]
\end{theorem}

Clearly, $b(n)\ge e(n)$ holds for all $n$, thus our results also bounds the function $e(n)$ defined for the original Erd\H os-Szekeres problem (cf. Theorem~\ref{suk}).

\begin{remark}
In our proof of Theorem \ref{gen conf ES}, for the sake of clarity, we do not focus on the constant in error-term in the bound on $b(n)$. However, with a less wasteful calculation (given at the end of the paper) we obtain the bound 
\[b(n)\leq 2^{n + \left( \frac{8\sqrt{2}}{3} +o(1) \right) \sqrt{n\log n}}.\]
\end{remark}

\paragraph{\bf Families of convex bodies.}
Bisztriczky and G. Fejes T\'oth \cite{bizft1,biszFT2} gave another (seemingly unrelated) generalization of the Erd\H os-Szekeres problem in 1989 by replacing point sets with families of pairwise disjoint convex bodies. They defined $n$ convex bodies to be in {\em convex position} if the convex hull of no $n-1$ of them contains the remaining one. If any three members of a family of convex bodies are in convex position, then the family is in {\em general position}. In their pioneering paper, Bisztriczky and Fejes T{\'o}th proved that for any $n\ge 3$, there exists a smallest integer $c(n)$ with the following property. If $\mathcal F$ is a family of pairwise disjoint convex bodies in general position in the plane with $|\mathcal F|\ge c(n)$, then it has $n$ members in convex position. They conjectured that $c(n)=e(n)$. The first singly-exponential upper bound on $c(n)$ was established by Pach and T\'oth~\cite{PachTothBodies}. They extended the statement to families of pairwise {\em noncrossing} convex bodies, that is, to convex bodies that may intersect, but any pair can share at most two boundary points~\cite{PachToth1}. This assumption is necessary.

\begin{theorem} [Pach--T\'oth, 2000]
For any integer $n\geq 3$, there exists a smallest number $c'(n)$ with the following property. Any family of at least $c'(n)$ pairwise noncrossing convex bodies in general position in the plane has $n$ members in convex position.
\end{theorem}

Clearly, we have $c'(n)\ge c(n)\ge e(n)$ for every $n$. The original upper bound on $c'(n)$ was subsequently improved by Hubard, Montejano, Mora, and  Suk \cite{hubsuk} and by Fox, Pach, Sudakov, and Suk \cite{FPSS} to $2^{O(n^2\log n)}$, and later by Dobbins, Holmsen, and Hubard~\cite{dobbs1} to $4^n$. More importantly from our point of view, they showed that there is an intimate relationship between the generalizations of the Erd\H os-Szekeres problem to non-crossing convex bodies and to pseudo-configurations. The following is the union of Lemmas~2.4 and~2.7 in their paper.

\begin{theorem}[Dobbins--Holmsen--Hubard, 2014] \label{pseudo to noncross}
Let $\mathcal F$ be a family of pairwise noncrossing convex bodies in general position in the plane. There exists a pseudo-configuration $P$ and a bijection $\varphi : P \to {\mathcal F}$ such that for any subset $S\subseteq P$ which is in convex position, the subfamily $\varphi(S)$ is also in convex position.
\end{theorem}

It follows from this result that $c'(n)\le b(n)$ for all $n$. (In fact, it was shown in \cite{dobbs1} that $c'(n) = b(n)$ for all $n\geq 3$.) In view of this, Theorem~\ref{gen conf ES} immediately implies the following.

\begin{theorem} \label{main bft}
Given any $n\geq 3$, let $c'(n)$ denote the smallest number such that every family of at least $c'(n)$ pairwise noncrossing convex bodies in general position in the plane has $n$ members in convex position. Then we have
\[c'(n) \leq 2^{n+O\left(\sqrt{n\log n}\right)}.\]
\end{theorem}

Here is a summary of the known  bounds on the various functions discussed above:
\[2^{n-2}+1 \leq e(n) \leq c(n) \leq c'(n) = b(n) \leq 2^{n+O(\sqrt{n\log n})}.\]
(Note that none of the inequalities are known to be strict.)

\medskip

The rest of this note is organized as follows. After highlighting two auxiliary results in Section~2, we present the proof of Theorem~\ref{gen conf ES} in Section~3. In the end of Section~3 we show how to optimize the constant appearing in the error-term.

\section{Auxiliary results}
To follow Suk's line of argument, we recall two results needed for the proof: a combinatorial version of the ``cup-cap'' theorem (Theorem~\ref{transitive}) and a variant of a positive fraction Erd{\H o}s--Szekeres theorem \cite{baranyES} (Theorem~\ref{clusterings}). For future reference, we also collect some simple observations on pseudo-configurations in convex positions (Observations~\ref{trivial} and~\ref{trivial2}).

\medskip

\paragraph{\bf Transitive colorings.} Let $S$ be a finite set with a given linear ordering $\prec$, and suppose the ordered triples $s_i \prec s_j \prec s_k$ are partitioned into two parts $T_1 \cup T_2$.  This partition is called a {\em transitive coloring} if every $s_1 \prec s_2 \prec s_3 \prec s_4$ in $S$ and $i\in\{1,2\}$ satisfy
\[(s_1,s_2,s_3),(s_2,s_3,s_4) \in T_i \Rightarrow  (s_1,s_2,s_4),(s_1,s_3,s_4) \in T_i. \]
Transitive colorings were introduced in \cite{FPSS} and \cite{hubsuk}. The following statement can be proved in precisely the same way as the ``cup-cap'' theorem; see \cite{moshkovitz} for an alternative proof.

\begin{theorem}\label{transitive}
\cite{FPSS, hubsuk} \; Let $S$ be a finite set with a given linear ordering and let $T_1\cup T_2$ be a transitive coloring of the triples of $S$. If
\begin{equation}\label{cupcapbound}
|S| > \binom{k+l-4}{k-2},
\end{equation}
then there exists a $k$-element subset $S_{1}\subseteq S$ such that every triple of $S_{1}$ is in $T_1$, or an $l$-element subset $S_{2}\subseteq S$ such that every triple of $S_{2}$ is in $T_2$.
\end{theorem}

\paragraph{\bf Convex hulls of pseudo-configurations.} Below we collect a few simple observations on the convexity structure of pseudo-configurations. These statements are trivial for the usual notion of convexity, and easy to prove in this more general context.

\begin{obs}\label{trivial}\renewcommand{\qedsymbol}{}
Let $P$ be a pseudo-configuration.
\begin{itemize}
\item[(i)]The convex hull is a monotone operation. That is, for any $X\subseteq Y\subseteq P$, we have ${\rm conv} X\subseteq{\rm conv} Y$.
\item[(ii)] ${\rm conv} X$ is a simply connected closed set, for any $X\subseteq P$.
\item[(iii)] If $X\subseteq P$ is in convex position, then all points of $X$ appear on the boundary of ${\rm conv} X$.
\item[(iv)] Let $k\ge3$, and assume that $X = \{x_1, x_2, \dots, x_k\} \subseteq P$ is in convex position, where the points $x_i$ appear on the boundary of ${\rm conv} X$ in this cyclic order. Then the boundary of ${\rm conv} X$ is the union of the pseudosegments ${\rm conv}\{x_i,x_{i+1}\}$ for $1\le i\le k$. Furthermore, for each $i$, the pseudoline $\ell(x_i,x_{i+1})$ intersects ${\rm conv} X$ in the pseudosegment ${\rm conv}\{x_i,x_{i+1}\}$, and (the rest of) ${\rm conv} X$ lies entirely on one side of $\ell(x_i,x_{i+1})$.  (Indices are understood modulo $k$.) \qed
\end{itemize}
\end{obs}

\paragraph{\bf Convex clusterings.} Consider the pseudo-configuration described in Observation~\ref{trivial}(iv): Let $X = \{x_1, x_2, \dots, x_k\}$ be a $k$-element subset of $P$ in convex position, where $k\ge 3$, and suppose that the points $x_i$ appear on the boundary of $\conv X$ in this cyclic order. We define the $i$-th {\em spike} of $X$, denoted by $S_{i}$, to be the open region consisting of the points of the plane separated from the interior of $\conv X$ by the pseudoline $\ell(x_i,x_{i+1})$, but not separated from $\conv X$ by $\ell(x_{i-1},x_i)$ and by $\ell(x_{i+1},x_{i+2})$. This is a connected region bounded by the pseudosegment $\conv\{x_i,x_{i+1}\}$ and by portions of the pseudolines $\ell(x_{i-1}, x_i)$ and $\ell(x_{i+1},x_{i+2})$. It is either a triangle-like bounded region or an unbounded region of three sides; see Fig.~\ref{spike figure}.

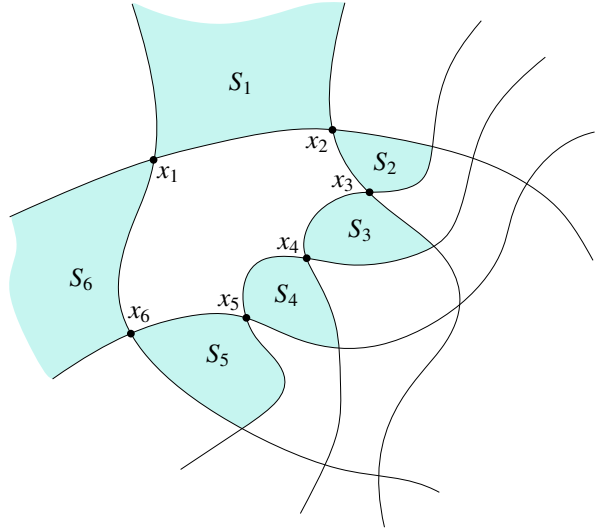
\begin{figure}\centering
\begin{tikzpicture}
\node at (0,0) {\includegraphics[height=7cm]{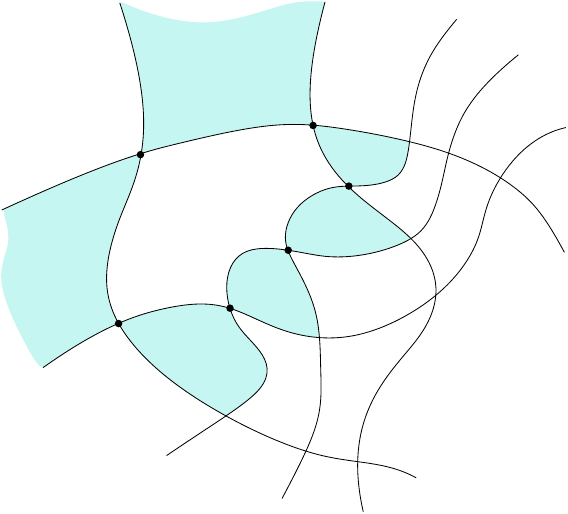}};
\node at (-1.76,1.2) {\fts $x_1$};
\node at (.2,1.6) {\fts $x_2$};
\node at (.54,1.08) {\fts $x_3$};
\node at (-.18,.26) {\fts $x_4$};
\node at (-.96,-.5) {\fts $x_5$};
\node at (-2.14,-.66) {\fts $x_6$};
\node at (-.8,2.4) {\fts $S_{1}$};
\node at (1.1,1.35) {\fts $S_{2}$};
\node at (.77,.45) {\fts $S_{3}$};
\node at (-.2,-.4) {\fts $S_{4}$};
\node at (-1.1,-1.2) {\fts $S_{5}$};
\node at (-2.9,-.2) {\fts $S_{6}$};
\end{tikzpicture}
\caption{A pseudo-configuration in convex position where the spikes are shaded.}
\label{spike figure}
\end{figure}

\begin{obs}\label{trivial2}\renewcommand{\qedsymbol}{}
Let $1\le i\le k$.
\begin{itemize}
\item[(i)] The line $\ell(x_i,x_{i+1})$ is disjoint from every spike and separates $S_i$ from all other spikes $S_j$ ($j\ne i$). In particular, the spikes are pairwise disjoint.
\item[(ii)] A point $p\in P\setminus X$ belongs to the spike $S_ i$ if and only if $X'=X\cup\{p\}$ is in convex position and $p$ appears on the boundary of ${\rm conv} X'$ between $x_i$ and $x_{i+1}$. In particular, whether $X\cup\{p\}$ is in convex position is determined by which region $p$ belongs in the arrangement of pseudolines spanned by $X$.   \qed
\end{itemize}
\end{obs}

For the usual notion of convexity in the Euclidean plane, the following statement was proved by P{\'o}r and Valtr \cite{partitionES}. It is a slight strengthening of a result of Pach and Solymosi~\cite{pach-soly1} that can be obtained by simple double counting. Since we will use this statement for pseudo-configurations, to make our paper self-contained, we translate its proof into this setting.

\begin{theorem} \label{clusterings}
Let $k\geq 3$ be an integer, and let $P$ be a pseudo-configuration with $|P| = N \geq 2^{4k}$. Then there exists a subset $X = \{x_1, x_2, \dots, x_k\} \subset P$ in convex position such that the sets $P_i$ of all points of $P$ lying in the $i$-th spike, $i=1,\ldots,k,$ satisfy the inequality
\begin{equation}\label{eq:cluster}
 \prod_{i=1}^k |P_i|\geq \frac{N^k}{2^{8k^2}}.
\end{equation}
\end{theorem}

\begin{proof}
Let $P$ be a pseudo-configuration with $|P| = N\geq 2^{4k}$. By Theorem \ref{weak pseudo ES}, every $4^{2k}$-element subset $Q\subseteq P$ contains a $2k$-element subset $R\subset Q$ in convex position. Therefore, by double-counting, $P$ has at least
\[\frac{\binom N{4^{2k}}}{\binom{N-2k}{4^{2k}-2k}} = \frac{\binom{N}{2k}}{\binom{4^{2k}}{2k}}>\left(\frac N{4^{2k}}\right)^{2k}\]
distinct $2k$-element subsets in convex position.

Given a $2k$-element subset $Y$ in convex position, we say that a $k$-element subset $X\subset Y$ {\em supports} $Y$ if the points of $Y$ along the boundary of $\conv Y$ alternately belong to $X$ and $Y\setminus X$. Note that $Y$ is supported by two subsets.

Since the number of $k$-element subsets of $P$ in convex position is at most $\binom{N}{k}$,
there exists a $k$-element subset $X$ which supports at least
\[\frac{\left(\frac N{4^{2k}}\right)^{2k}}{\binom Nk} > \frac{N^k}{2^{8k^2}}\]
distinct $2k$-element subsets in convex position. By Observation~\ref{trivial2}(ii), if $X$ supports $Y$, then the points of $Y\setminus X$ belong to distinct spikes of $X$, which implies inequality \eqref{eq:cluster}.
\end{proof}

\section{Proof of Theorem \ref{gen conf ES}}

Consider a sufficiently large fixed pseudo-configuration $P$, let $k\ge 4$ be an even integer, and let $X = \{x_1, x_2, \dots, x_k\}\subset P$ be a $k$-element subset in convex position such that its points appear on the boundary of $\conv X$ in this cyclic order. Suppose that $X$ meets the requirements of Theorem \ref{clusterings}. As before, let $S_{1}, S_{2}, \dots, S_{k}$ denote the spikes of $X$ and let $P_i = P \cap S_{i}$. The indices are taken modulo $k$.

\paragraph{\bf Vertical and horizontal orderings on ${P_i}$.} Let $p$ and $q$ be distinct points in $P_i$. We write
\[p\prec_i^v q \iff \conv \{x_{i-1},p,x_{i+2}\} \subset \conv \{x_{i-1},q,x_{i+2}\},\]
\[p\prec_i^h q \iff \conv\{x_{i-1},q\} \cap \conv \{x_{i+2},p\} \neq \emptyset,\]
where the superscripts $v$ and $h$ refer to the adjectives ``vertical'' and ``horizontal'', respectively.

\begin{obs} \label{orthogonality}
Let $1\le i\le k$.
\begin{itemize}
\item[(i)] Both $\prec_i^v$ and $\prec_i^h$ are partial orders on $P_i$.
\item[(ii)] Any two distinct elements of $P_i$ are comparable by either $\prec_i^v$ or $\prec_i^h$, but not by both.
\end{itemize}
\end{obs}

\begin{proof}
The definition of $\prec_i^v$ clearly implies that it is a partial order. To see that the same is true for $\prec_i^h$, one has to show that if $p\prec_i^hq\prec_i^hr$ for three points $p,q,r\in P_i$, then $p\prec_i^hr$. This can be done by checking the few possible pseudo-configurations of the five points $x_{i-1}$, $x_{i+2}$, $p$, $q$ and $r$.

To prove (ii), it is sufficient to consider the pseudo-configurations consisting of only four points: $x_{i-1}$, $x_{i+2}$, and two points $p$ and $q$ from $P_i$. Using the fact that $p$ and $q$ lie on the same side of $\ell(x_{i-1}x_{i+2})$, one can show that out of the four relations $p\prec_i^v q$, $q\prec_i^v p$, $p\prec_i^h q$, and $q\prec_i^h p,$ precisely one will hold. Consider the four open regions into which the pseudolines $\ell(px_{i-1})$ and $\ell(px_{i+2})$ partition the plane. The region in which $q$ lies uniquely determines which of the above four relations will hold.
\end{proof}

For $1\leq i \leq k$, let $v_i$ denote the length of the longest chain in $P_i$ with respect to $\prec^v_i$, and let $h_i$ denote the length of the longest chain in $P_i$ with respect to $\prec^h_i$. By Observation \ref{orthogonality} and by (the easy part of) Dilworth's theorem, we have
\begin{equation} \label{eq:dilworth}
|P_i| \leq v_i h_i.
\end{equation}

\paragraph{\bf Further observations concerning points and spikes.} As before, the following observations are trivial for the usual notion of convexity in the Euclidean plane. Here we show that they also hold for pseudo-configurations.

\begin{obs} \label{twospikes}
For any pair of distinct points $p,q\in P$, the pseudoline $\ell(p,q)$ intersects at most two spikes of $X$.
\end{obs}

\begin{proof}
Assume for contradiction that $\ell(p,q)$ intersects three separate spikes $S_i$, $S_j$, and $S_l$ in this order. By Observation~\ref{trivial2}(i), this line should intersect $\ell(x_j,x_{j+1})$ twice, a contradiction.
\end{proof}

\begin{obs}\label{vertical separation}
Let $p$ and $q$ be distinct points of $P_i$. If $p\prec_i^v q$, then the pseudoline $\ell(p,q)$ separates spikes $S_{i-1}$ and $S_{i+1}$.
\end{obs}

\begin{proof}
Since $p\in \conv \{x_{i-1},q,x_{i+2}\}$, the pseudoline $\ell(p,q)$ intersects the pseudosegment $\conv \{x_{i-1},x_{i+2}\}$. This implies that $\ell(p,q)$ has to intersect one of the spikes  $S_{i+2}$, $S_{i+3}$, $\dots$, $S_{i-2}$. By Observation \ref{twospikes}, $\ell(p,q)$ intersects at most two spikes, one of which is $S_{i}$. Thus, it cannot intersect $S_{i-1}$ and $S_{i+1}$, which implies that $S_{i-1}$ and $S_{i+1}$ must be separated by $\ell(p,q)$.
\end{proof}

\begin{obs}\label{horizontal separation}
Let $p$ and $q$ be distinct points of $P_i$. If $p\prec_i^h q$, then all the spikes $S_{i+2}, S_{i+3}, \dots, S_{i-2}$ must lie on the same side of the pseudoline $\ell(p,q)$.
\end{obs}
\begin{proof}
All spikes $S_j$ with $j\notin\{i-1,i+1\}$ are on the same side of both pseudolines $\ell(x_{i-1},x_i)$ and $\ell(x_{i+1},x_{i+2})$. The angular region determined by these two pseudolines and containing the above spikes (and the interior of $\conv X$) is cut into two parts by the pseudosegment $\conv\{x_{i-1},x_{i+2}\}$, so that $S_i$ lies on one side and the spikes $S_ j$ with $j\notin\{i-1,i,i+1\}$ on the other. Our assumption $p\prec_i^h q$ implies that the pseudoline $\ell(p,q)$ does not intersect the pseudosegment $\conv\{x_{i-1},x_{i+2}\}$, so the part of the angular region on the other side of this pseudosegment (including all relevant spikes) is on the same side of $\ell(p,q)$, as claimed.
\end{proof}

\begin{figure}\centering
\begin{tikzpicture}
\node at (0,0) {\includegraphics[height=6.3cm]{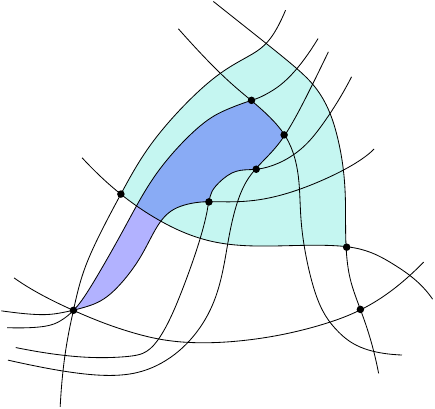}};
\node at (-1.98,-1.92) {\fts $x_{i-1}$};
\node at (-1.76,.13) {\fts $x_i$};
\node at (2.33,-.53) {\fts $x_{i+1}$};
\node at (2.59,-1.76) {\fts $x_{i+2}$};

\node at (.03,-.2) {\fts $p_1$};
\node at (.74,.3) {\fts $p_2$};
\node at (1.3,1.1) {\fts $p_3$};
\node at (.6,1.87) {\fts $p_4$};
\end{tikzpicture}
\caption{A left convex chain $p_1 \prec_i^v p_2 \prec_i^v p_3 \prec_i^v p_4$ in $P_i$ with $\conv\{p_1,p_2,p_3,p_4,x_{i-1}\}$ in darker shade.}
\label{left-convex}
\end{figure}
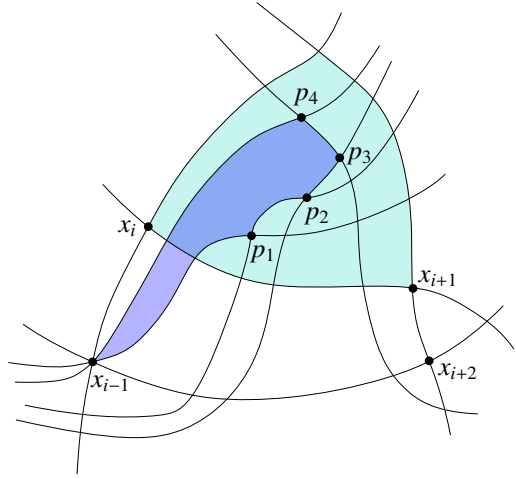

\paragraph{\bf Vertical convex chains} Let $C \subseteq P_i$ be a chain with respect to $\prec_i^v$. If $\{x_{i-1}\} \cup C$ is in convex position, we call $C$ a {\em left convex chain} in $P_i$. If $\{x_{i+2}\}\cup C$ is in convex position, we call $C$ a {\em right convex chain} in $P_i$.

Note that if $|C|=3$, then $C$ is either a left convex chain or a right convex chain, but not both. This can be verified by checking the pseudo-configuration $C\cup\{x_{i-1},x_{i+2}\}$. Moreover, if we have $p_1\prec_i^v p_2 \prec_i^v p_3 \prec_i^v p_4$ and both $\{p_1, p_2, p_3\}$ and $\{p_2,p_3,p_4\}$ are left (right) convex chains, then $\{p_1,p_2,p_3,p_4\}$ is also a left (right) convex chain. Therefore, the same holds for both $\{p_1,p_2,p_4\}$ and $\{p_1,p_3,p_4\}$. This can be verified by checking the pseudo-configuration $\{p_1,p_2,p_3,p_4, x_{i-1}, x_{i+2}\}$. See Fig.~\ref{left-convex}.

Take a chain $C\subseteq P_i$ of maximal size $|C|=v_i$, totally ordered by $\prec_i^v$. Partition the triples of $C$ into left and right convex chains. In this way, we obtain a transitive coloring. Letting $a_i$ and $b_i$ denote the length of the longest left convex chain and the length of the longest right convex chain in $C$, respectively, by Theorem~\ref{transitive}, we have
\begin{equation} \label{vertical-binomial}
v_i \leq \binom{a_i+b_i-2}{a_i-1}.
\end{equation}
Actually, Theorem~\ref{transitive} only guarantees the existence of large subsets $C_1, C_2\subseteq C$ such that all triples in $C_1$ are left convex chains and all triples in $C_2$ are right convex chains. However, using the above observations and the generalization of Carath\'eodory's theorem to pseudo-configurations, it follows that $C_1$ and $C_2$ themselves must form a left convex chain and a right convex chain, respectively.

\begin{obs} \label{joining verticals}
If $R$ is a right convex chain in $P_{i}$ and $L$ is a left convex chain in $P_{i+1}$, then $R\cup L$ is in convex position.
\end{obs}

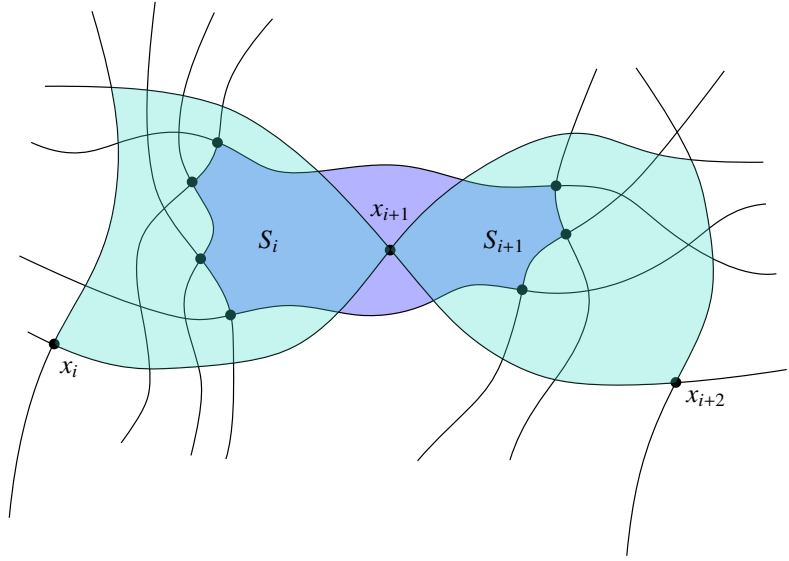
\begin{figure}\centering
\begin{tikzpicture}
\node at (0,0) {\includegraphics[height=7.4cm]{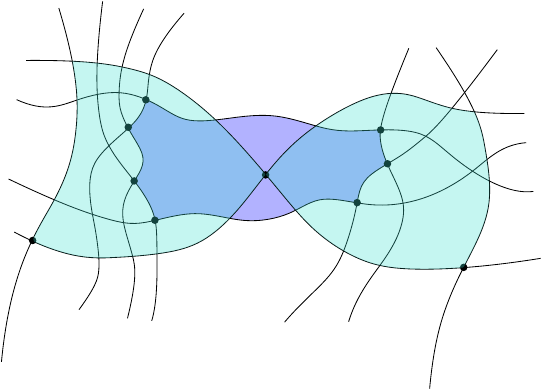}};
\node at (-.1,.89) {\fts $x_{i+1}$};
\node at (-4.33,-1.2) {\fts $x_{i}$};
\node at (4.05,-1.58) {\fts $x_{i+2}$};
\node at (-1.7, .5) {\fts $S_{i}$};
\node at (1.4,.5) {\fts $S_{i+1}$};
\end{tikzpicture}
\caption{Joining a right convex chain $R\subseteq P_i$ and a left convex chain $L\subseteq P_{i+1}$ to form a subset in convex position (convex hull in darker shade).}
\label{vertical-convex}
\end{figure}

\begin{proof}
First, note that for any pseudo-configuration $P$ consisting of four points, if a point $p\in P$ lies in the convex hull of $P\setminus \{p\}$, then any pseudoline passing through $p$ and any other point of $P$ crosses the pseudosegment determined by the other two points of $P$.

To prove the observation, it is enough to show that any four points $p,q,r,s \in R\cup L$ are in convex position. If all of them lie in one of $R$ or $L$, then we are clearly done. Assume first that $r,s\in R$ and $p,q\in L$. By Observation \ref{vertical separation}, the pseudolines $\ell(p,q)$ and $\ell(r,s)$ do not intersect the pseudosegments $\conv\{r,s\}\subset S_{i}$ and $\conv\{p,q\}\subset S_{i+1}$, respectively. Therefore, by the discussion above, the points $p,q,r,s$ are in convex position.

Now consider the case where $p,q,r\in L$ and $s\in R$. Again by Observation \ref{vertical separation}, none of the pseudolines $\ell(p,q)$, $\ell(p,r)$, and $\ell(q,r)$ intersects the spike $S_{i}$. Therefore, $x_i$ and $s$ lie in the same open region determined by the arrangement of these three pseudolines. By the assumption, the set $\{p,q,r,x_i\}$ is in convex position, so by the last statement of Observation \ref{trivial2}(ii) $\{p,q,r,s\}$ is in convex position, as well. The other case, $ p\in L$ and $q,r,s \in R$, can be settled in a similar manner. See Fig. \ref{vertical-convex}.
\end{proof}

\paragraph{\bf Horizontal convex chains.} Let $C\subseteq P_i$ be a chain with respect to $\prec_i^h$. If $\{p,q,r,x_{i-1},x_{i+2}\}$ is in convex position for any three distinct elements $p,q,r$ of $C$, we call $C$ an {\em inner convex chain}. If $\{p,q,r,x_{i-1},x_{i+2}\}$ is {\em not} in convex position for any three distinct elements $p,q,r$ of $C$, we call $C$ an {\em outer convex chain}.

Note that chains of at most two elements are both inner and outer convex chains by this definition.

\begin{obs} \label{inout}
Let $1\leq i \leq k$.
\begin{itemize}
\item[(i)] The partitioning of the triples in a horizontal chain $C$ (ordered by $\prec_i^h$) into inner and outer convex chains is a transitive coloring.
\item[(ii)] The inner and outer convex chains in $P_i$ are in convex position.
\end{itemize}
\end{obs}

\begin{proof}
Consider a horizontal chain $p\prec_i^h q \prec_i^h r$ in $P_i$. By checking the pseudo-configuration $\{p,q,r,x_{i-1},x_{i+2}\}$ we can verify that the following are equivalent:
\begin{itemize}
\item $(p,q,r)$ is an outer (inner) convex chain.
\item $\conv \{x_{i-1}, x_{i+2}\}$ and $r$ are separated by (lie on the same side of) $\ell(p,q)$.
\item $\conv \{x_{i-1}, x_{i+2}\}$ and $p$ are separated by (lie on the same side of) $\ell(q,r)$.
\item $\conv \{x_{i-1}, x_{i+2}\}$ and $q$ lie on the same side of (are separated by) $\ell(p,r)$.
\end{itemize}

Now consider a horizontal chain $p_1 \prec_i^h p_2 \prec_i^h p_3 \prec_i^h p_4$. The pseudolines $\ell(x_{i-1},p_4)$ and $\ell(x_{i+2}, p_1)$ divide the plane into four quadrants, each containing one of the pseudosegments $\conv\{p_1,p_4\}$, $\conv\{p_4,x_{i+2}\}$, $\conv\{x_{i+2},x_{i-1}\}$, $\conv\{x_{i-1},p_1\}$, in this cyclic order. By the ordering $\prec_i^h$, $p_2$ and $p_3$ are contained in the quadrant containing $\conv\{p_1,p_4\}$. Furthermore, the pseudoline $\ell(p_2,p_3)$ must cross this quadrant, entering the boundary ray containing $p_1$, then meeting $p_2$ before $p_3$ and finally exiting the boundary ray containing $p_4$. If both $(p_1,p_2,p_3)$ and $(p_2,p_3,p_4)$ are outer (inner) convex chains, it follows by the observations above that $\conv\{p_1,p_4\}$ and $\conv\{x_{i-1},x_{i+2}\}$ are separated by (lie on the same side of) $\ell(p_2,p_3)$. This implies that $\conv\{x_{i-1},x_{i+2}\}$ and $p_4$ are separated by (lie on the same side of) $\ell(p_1,p_2)$ and $\ell(p_1,p_3)$. Hence,  $(p_1,p_2,p_4)$ and $(p_1,p_3,p_4)$ are both outer (inner) convex chains, which proves part (i). By Carath{\'e}odory's theorem, it suffices to check part (ii) for inner and outer convex chains $p_1 \prec_i^h p_2 \prec_i^h p_3 \prec_i^h p_4$. However, it follows from the discussion above that $\ell(p_1,p_4)$ does not intersect $\conv\{p_2,p_3\}$ and that $\ell(p_2,p_3)$ does not intersect $\conv \{p_1,p_4\}$. As in the proof of Observation \ref{joining verticals}, we obtain that $\{p_1, p_2, p_3, p_4\}$ is in convex position. See Fig.~\ref{outer-convex}.
\end{proof}

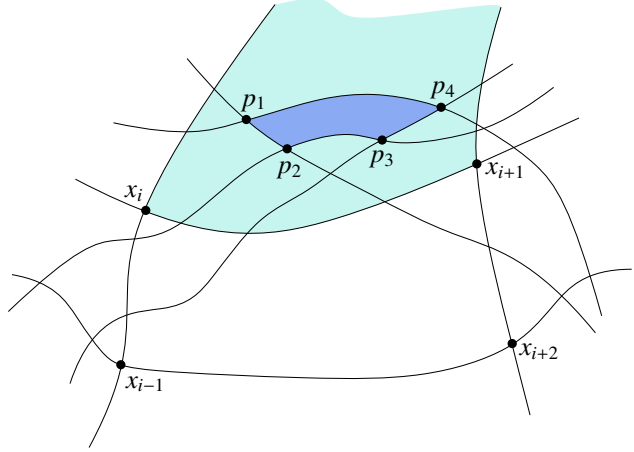
\begin{figure}\centering
\begin{tikzpicture}
\node at (0,0) {\includegraphics[height=6cm]{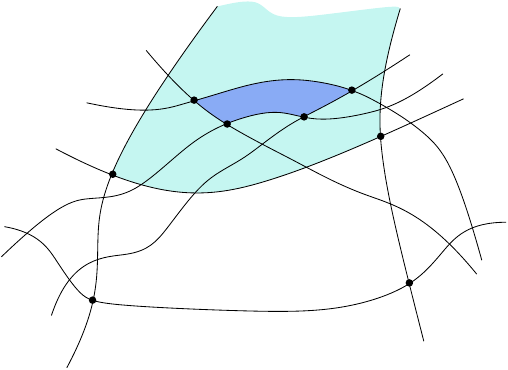}};
\node at (-2.3,-2.13) {\fts $x_{i-1}$};
\node at (-2.45,.4) {\fts $x_{i}$};
\node at (2.44,.7) {\fts $x_{i+1}$};
\node at (2.88,-1.7) {\fts $x_{i+2}$};
\node at (-.87,1.6) {\fts $p_1$};
\node at (-.4,.7) {\fts $p_2$};
\node at (.82,.85) {\fts $p_3$};
\node at (1.62,1.76) {\fts $p_4$};
\end{tikzpicture}
\caption{An outer convex chain $p_1 \prec_i^h p_2 \prec_i^h p_3 \prec_i^h p_4$ in $P_i$ with the $\conv\{p_1,p_2,p_3,p_4\}$ in darker shade. }
\label{outer-convex}
\end{figure}

Letting $c_i$ and $d_i$ denote the length of the longest inner convex chain and the length of the longest outer convex chain in $P_i$, respectively, applying Theorem~\ref{transitive} to the longest horizontal chain in $P_i$ and using Observation~\ref{inout}, we obtain
\begin{equation}\label{horizontal-binomial}
h_i \leq \binom{c_i + d_i -2}{c_i-1}.
\end{equation}

\begin{obs} \label{joining horizontals}
Suppose that $k\ge4$ is even, and let $A_1 \subseteq P_1$, $A_2\subseteq P_2$, $\dots$, $A_k\subseteq P_k$. If each $A_i$ is an inner convex chain, then $A_1 \cup A_3 \cup \cdots \cup A_{k-1}$ is in convex position, and so is $A_2 \cup A_4 \cup \cdots \cup A_k$.
\end{obs}

\begin{proof}
The proof follows in the same way as Observation \ref{joining verticals}, and we repeatedly use the fact mentioned at the beginning of that proof. It suffices to prove that any four points $p_1,p_2,p_3,p_4 \in A_1 \cup A_3 \cup \cdots \cup A_{k-1}$ are in convex position. If all the points lie in one chain, we are done. Consider the case where three points belong to the same chain, say $p_1,p_2,p_3\in A_{i_1}$ and $p_4\in A_{i_2}$ with $i_1\neq i_2$. By Observation \ref{horizontal separation}, $x_{i_1-1}$ and $p_4$ belong to the same open region determined by the pseudolines $\ell(p_1,p_2), \ell(p_1,p_3), \ell(p_2,p_3)$. Therefore, by the last statement of Observation \ref{trivial2}(ii), the convexity of $\{p_1, p_2, p_3, x_{i_1-1}\}$ implies that $\{p_1, p_2, p_3, p_4\}$ is in convex position.

If one of the chains contains exactly two of our points, say $p_1,p_2 \in A_i$, then neither $p_1$ nor $p_2$ can be in the convex hull of the other three points, as Observation~\ref{horizontal separation} implies that the pseudoline $\ell(p_1,p_2)$ does not intersect the pseudosegment $\conv\{p_3,p_4\}$.

\begin{figure}\centering
\begin{tikzpicture}
\node at (0,0) {\includegraphics[height=9.7cm]{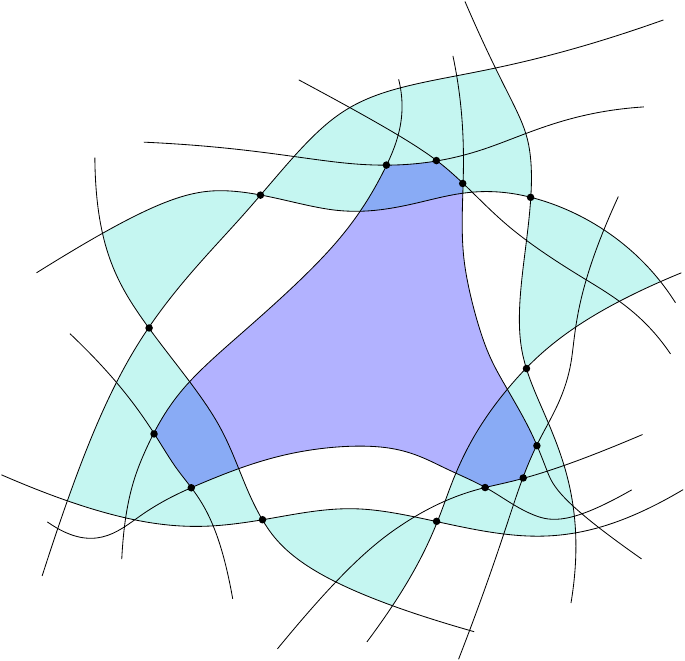}};
\node at (-1.15,1.7) {\fts $x_1$};
\node at (2.48,1.8) {\fts $x_2$};
\node at (2.47,-.5) {\fts $x_3$};
\node at (1.54,-3.06) {\fts $x_4$};
\node at (-1.26,-3.03) {\fts $x_5$};
\node at (-2.54,0.02) {\fts $x_6$};
\end{tikzpicture}
\caption{Joining inner convex chains $A_1\subseteq P_1, A_3\subseteq P_3$, and $A_5\subseteq P_5$ to form a subset in convex position (convex hull in darker shade).}
\label{horizontal-convex}
\end{figure}
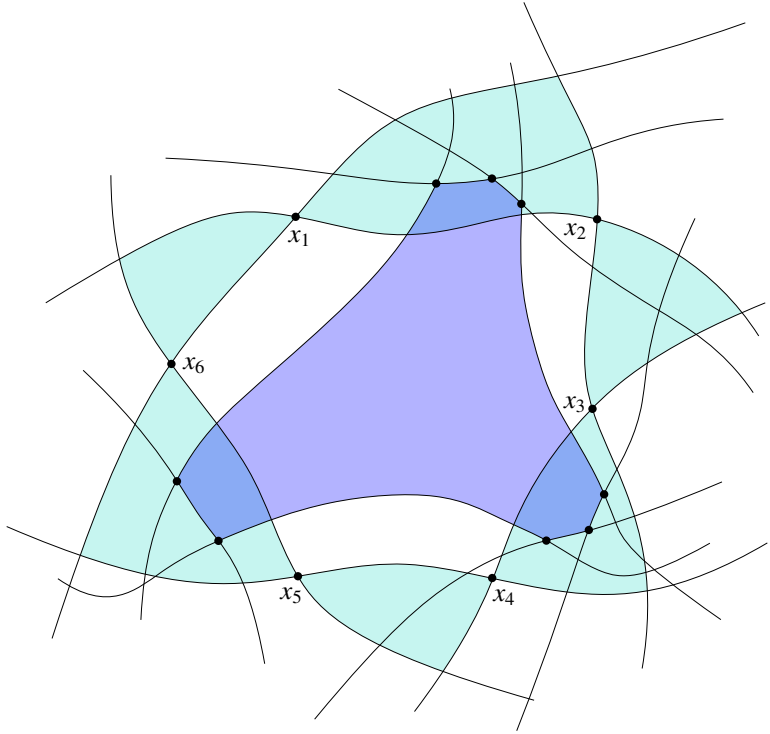

To finish the proof, we need to verify that if one of the chains contains exactly one of our points, say $p_1\in A_i$, then $p_1$ is not in the convex hull of the other three points. This follows from the fact that $\ell(x_i,x_{i+1})$ separates $p_1$ from $p_2$, $p_3$ and $p_4$. See Fig.~\ref{horizontal-convex}.
\end{proof}

\paragraph{\bf Proof of Theorem \ref{gen conf ES}}
Let $P$ be a pseudo-configuration on $N$ points, and suppose that $P$ does not contain $n$ points in convex position. Let $k$ be an even integer to be specified later, and let $X = \{x_1, x_2, \dots, x_k\}\subseteq P$ be a subset in convex position whose existence is guaranteed by Theorem \ref{clusterings}. 
As above, for the sets $P_i$ of all points of $P$ contained in the $i$-th spike of $X$, $i=1,\dots, k$,  we define
\begin{itemize}
    \item[] $v_i$ -- the length of the longest chain $C^v_i$ with respect to $\prec^v_i$.
    \item[] $h_i$ -- the length of the longest chain $C^h_i$ with respect to $\prec^h_i$.
    \item[] $a_i$ -- the length of the longest left convex chain in $C^v_i$.
    \item[] $b_i$ -- the length of the longest right convex chain in $C^v_i$.
    \item[] $c_i$ -- the length of the longest inner convex chain in $C^h_i$.
    \item[] $d_i$ -- the length of the longest outer convex chain in $C^h_i$.
\end{itemize}

By Observation~\ref{inout}(ii), we have $d_i <n$. By Observation \ref{joining verticals}, we have
\begin{equation}
b_i + a_{i+1} < n
\end{equation}
for all $i$, and, by Observation~\ref{joining horizontals},
\begin{equation}
c_1+c_2+\cdots + c_k < 2n.
\end{equation}
Combining these with inequalities \eqref{eq:cluster}--\eqref{horizontal-binomial}, we obtain

\begin{eqnarray*}
\frac{N^k}{2^{8k^2}} & \leq & \prod_{i=1}^k|P_i|  \\
& \leq & \prod_{i=1}^k v_ih_i  \\ &\leq & \prod_{i=1}^k \binom{a_i+b_i-2}{a_i-1} \binom{c_i+d_i-2}{c_i-1} \\
& < & \prod_{i=1}^k 2^{a_i+b_i} 
d_i^{c_i} < 2^{kn+2n\log n},
\end{eqnarray*}
which gives us
\begin{equation*}
N < 2^{n+\frac{2n\log n}{k} +8k}.
\end{equation*}
Setting $k$ to be the smallest even integer greater than or equal to $\frac{1}{2}\sqrt{n\log n}$, gives the estimate
\[\pushQED{\qed}
N=O\left(2^{n+8\sqrt{n\log n}}\right).\qedhere
\popQED
\]

\paragraph{\bf Optimizing the error term.} Here we improve the error term in our previous estimate by showing the bound 
\[b(n) \leq 2^{n + \left(\frac{8\sqrt{2}}{3}+o(1)\right)\sqrt{n\log n}}.\]
The first improvement is a refinement of Theorem \ref{clusterings}.
\begin{prop}\label{refinement} Let $k\geq 3$ be an integer, and let $P$ be a pseudo-configuration with $|P| = N \geq 2^{\left(1+o(1)\right)4k}$. Then one of the following hold.
\begin{enumerate}
    \item \label{2koverk} There exists a subset $X = \{x_1, x_2, \dots, x_k\} \subset P$ in convex position such that the sets $P_i$ of all points of $P$ in the $i$-the spike of $X$, $i=1, 2, \dots, k$, satisfy
    \[\prod_{i=1}^k |P_i| \geq 2^{-\frac{8}{3}k^2}N^k.\]
    \item \label{4kover2k} There exists a subset $X' = \{x_1', x_2', \dots, x_{2k}'\} \subset P$ in convex position such that the sets $P'_i$ of all points of $P$ in the $i$-the spike of $X'$, $i=1, 2, \dots, 2k$, satisfy
    \[\prod_{i=1}^{2k}|P'_i| \geq 2^{-\frac{40}{3}k^2 - o(k^2)}N^{2k}.\]
\end{enumerate}
\end{prop}
\begin{proof}
Let $f_j = f_j(P)$ denote the number of $j$-element subset of $P$ that are in convex position. Looking back at the proof of Theorem \ref{clusterings}, we see that for the optimal set $X$, the quantity $\prod_{i=1}^k|P_i|$ is bounded below by $f_{2k}/f_k$. Similarly we have $\prod_{i=1}^{2k}|P'_i|\geq f_{4k}/f_{2k}$ for the optimal set $X'$. Trivially we have $f_k \leq N^k$, and using our (preoptimized) bound on $b(n)$, with the same double-counting as before we have  
$f_{4k} \geq 2^{-16k^2 - o(k^2)} N^{4k}$. The claim now follows from whether or not $f_{2k} \geq 2^{-\frac{8}{3}k^2} N^{2k}$. 
\end{proof}

Here is another improvement. In the proof of Theorem \ref{gen conf ES} we used the estimate $\binom{c_i+d_i-2}{d_i-1} < d_i^{c_i}$, where $d_i<n$ and $\sum_ic_i<2n$, which gave us $\prod_i \binom{c_i+d_i-2}{d_i-1} < n^{2n}$. Instead, if we use the more precise estimate
\[
\binom{c_i+d_i-2}{c_i-1} < \binom{2n}{c_i} <  \left( \frac{2en}{c_i} \right)^{c_i} 
\leq k^{c_i}e^{2n/k},
\] then we get
\begin{equation} \label{2ndimprov}
\prod_{i=1}^k \binom{c_i+d_i-2}{c_i-1} < 
(ek)^{2n}. \end{equation}

\medskip

Now we combine these two improvements. Let $P$ be a pseudo-configuration on $N$ points and suppose $P$ does not contain $n$ points in convex position. We apply the same argument as before to each of the cases in Proposition \ref{refinement}, using the better estimate from \eqref{2ndimprov}. In case \eqref{2koverk} we obtain the inequality 
\[2^{-\frac{8}{3}k^2}N^k \leq \prod_{i=1}^k |P_i| < 2^{kn+2n\log (ek)},\]
and in case \eqref{4kover2k} we obtain the inequality
\[2^{-\frac{40}{3}k^2 - o(k^2)}N^{2k} \leq \prod_{i=1}^{2k}|P'_i| < 2^{2kn + 2n\log (2ek)}.\] Setting the $k$ to be the smallest even integer greater than or equal to $\frac{\sqrt{n\log n}}{2\sqrt{2}}$, either case gives us the desired  bound $N< 2^{n + \left( \frac{8\sqrt{2}}{3} +o(1) \right)\sqrt{n\log n}}$.

\bigskip
\footnotesize
\noindent\textit{Acknowledgments.}

A.~F.~Holmsen partially  
supported by the National Research Foundation of Korea (NRF) grant funded by the Korean government (Ministry of Science and ICT) (No. 2020R1F1A1A0104849011).

H.~N.~Mojarrad supported by Swiss National Science Foundation grant P2ELP2\_178313.

J.~Pach partially supported by the National Research, Development and Innovation Office, NKFIH, project KKP-133864, the Austrian Science Fund (FWF), grant Z 342-N31 and by the Ministry of Education and Science of the Russian Federation in the framework of MegaGrant No.\ 075-15-2019-1926.

G.~Tardos partially supported by the Cryptography ``Lend\"ulet'' project of the Hungarian
Academy of Sciences and by the National Research, Development and Innovation Office (NKFIH)
grants K-116769, K-132696 and KKP-133864, by the ERC Synergy Grant ``Dynasnet'' No. 810115 and by the Ministry of Education and Science of the Russian Federation in the framework of MegaGrant No.\ 075-15-2019-1926.

\end{document}